\documentclass[11pt]{article}
\usepackage{amsmath,amsfonts,amsthm,amssymb, mathtools}
\usepackage{mathrsfs, graphicx,color,latexsym, tikz, calc, booktabs
}

\newtheorem{thm}{Theorem}[section]
\newtheorem{lem}[thm]{Lemma}

\newtheorem{prop}[thm]{Proposition}

\newtheorem{prob}{Problem}

\title{\bf \Large  A best bound for $\lambda_2(G)$ to  guarantee $\kappa(G) \geq 2$}
\author{
{\small  Wenqian Zhang, \ \ Jianfeng Wang\footnote{Corresponding author.
\newline{\it \hspace*{5mm}Email addresses:} zhangwq@sdut.edu.cn(W.Q. Zhang), jfwang@sdut.edu.cn (J.F.Wang).}}\\[2mm]
\footnotesize School of Mathematics and Statistics, Shandong University of Technology, Zibo 255049, China}
\date{ }
\date{}
\begin{document}
\maketitle

\setcounter{page}{1}
\begin{abstract}
Let $G$ be a connected $d$-regular graph  with a given order and  the second largest eigenvalue $\lambda_2(G)$. Mohar and O (private communication) asked a challenging problem: what is the best upper bound for $\lambda_2(G)$ which guarantees that $\kappa(G) \geq t+1$, where $1 \leq t \leq d-1$ and $\kappa(G)$ is the vertex-connectivity of $G$, which was also mentioned by Cioab\u{a}. As a starting point, we solve this problem in the case $t =1$, and characterize all families of extremal graphs. \\

\end{abstract}

\section{\Large Introduction}

In this paper, we focus on the eigenvalues $d > \lambda_2(G) \geq \cdots \geq \lambda_n(G)$ of adjacency matrix $A(G)$ of $d$-regular graphs with order $n$. Set $\lambda = \lambda(G) = \max\{|\lambda_2(G)|,|\lambda_n(G)|\}$ to be the {\it second eigenvalue} of $G$.   Proverbially, the second largest eigenvalue $\lambda_2$ and the second eigenvalue $\lambda$ of  graphs have been paid much attention in several aspects. One of the best-known cases is that it could describe the {\it expander graphs} which are useful in the design and analysis of communication networks. From the point view of spectral graph theory, it was  Alon and Milman \cite{alon-mil} and Dodziuk \cite{dod}  who independently gave a discrete analogue of Cheeger's result on Riemannian manifold \cite{che}, which is stated as follows and involves the {\it edge  expansion ratio $h(G)$ of $G$} .
\begin{equation}\label{A-B-bound}
\frac{d-\lambda}{2} \leq h(G) \leq \sqrt{2d(d-\lambda)}, \quad \mbox{where} \quad h(G) = \min_{S \subseteq V, |S| \leq |V|/2}\frac{|E(S,\overline{S})|}{|S|}.
\end{equation}
As we have known, large spectral gap $d-\lambda$ implies high expansion, vice versa. Remarkably, Alon \cite{alon-1} showed that a regular bipartite graph is an expander iff $d-\lambda$ is large enough (note $\lambda = \lambda_2$ for bipartite graphs). Further, we precisely know how large the spectral gap can be in an $d$-regular graph $G_{n,d}$ of order $n$, due to Alon-Boppana bound $\lim\limits_{n\rightarrow \infty}\inf\lambda(G_{n,d}) = 2\sqrt{d-1}$ (see \cite{nil}). See the survey \cite{hoo-lin-wig} for more details about this topic.

In light of the  Alon-Boppana bound \eqref{A-B-bound}, it opens the door to the second aspect which is related to the graphs characterized by the second (largest) eigenvalues. Recall, A $d$-regular graph $G$ is {\it Ramanujan} if $\lambda(G) \leq 2\sqrt{d-1}$. The significant progress in this field may be that there exist infinite families of regular bipartite Ramanujan
graphs of every degree bigger than 2, owe to Marcus et al \cite{mar-spi-sri}. On the other side, it is also a rather difficult problem that determining all the connected (regular) graphs with small $\lambda_2$ ($\leq 2\sqrt{d-1}$).  Recall, the graphs with $\lambda_2 \leq 2$ is the so called {\it reflexive graphs}  corresponding to sets
of vectors  norm 2 and at angles $90^{\circ}$ or $120^{\circ}$ but without a structural characterization, thanks to Neumaier and Seidel \cite{neu-sei}. Especially,  Maxwell \cite{max} identified such trees named as {\it Hyperbolic trees}. For other miscellaneous results about the graphs with small $\lambda_2 \leq \frac{1}{3}$ or $\sqrt{2}-1$ or $\frac{\sqrt{5}-1}{2}$, see Cvetkovi\'c and Simi\'c's survey \cite{cve-sim-filomat} and  \cite[eg.]{tam-sta,sta} for (regular) graphs with $\lambda_2 \leq 2$.

It is the third aspect that the second largest eigenvalues can be employed to depict the parametric properties of (regular) graphs. Obviously,  the edge expansion ratio in \eqref{A-B-bound}, as a graphs parameter, seems to be the first example. Remarkably, it has always been a hot topic describing the connections between the second largest eigenvalue of a regular graph and other combinatorial parameters, such as the toughness \cite[eg.]{alo-2,bro,cio-CMJ,cio-wong-DAM,liu-chen},  the spanning trees \cite[eg.]{cio-CMJ,cio-won-LAA,gu-lai-JGT}, the (perfect) matching \cite[eg.]{cio-JCTB,O-cio}, the regular factors \cite[eg.]{lu-JGT}, the matching extendability \cite{zhang}, the edge-connectivity \cite[eg.]{abi-bri-mar-O,cha-IPL,cio-LAA,kri-sud,O-LAA,O-cio} and so on. See \cite[Chapter 2]{bro-hae} for other details about the second (largest) eigenvalues.

In this paper, we pay attention to the second largest eigenvalue and vertex-connectivity of regular graphs. Unlike more fruitful results about the edge-connectivity $\kappa'(G)$ of a $d$-regular graph based on its second largest eigenvalue $\lambda_2(G)$, there are few ones concerning the vertex-connectivity $\kappa(G)$ built on $\lambda_2(G)$, the earliest two of which are respectively Krivelevich and Sudakov's Theorem 4.1 in \cite{kri-sud} and Fiedler's Theorem 4.1 \cite{fie} (note, $\lambda_2(G) = d-\alpha(G)$ with $\alpha(G)$ being the algebraic connectivity). The others are obtained in past five years. Some of them is fit for the regular multigraph \cite{abi-bri-mar-O,O-DAM}. Here we  concentrate on the results for the regular simple graphs, the first one of which is  due to Cioab\u{a} and Gu \cite{cio-CMJ}.

\begin{prop}{\rm \cite[Cioab\u{a} and Gu]{cio-CMJ}}
Let $G$ be a connected $d$-regular simple graph, $d \geq 3$, and
$$\lambda_2(G) <
\begin{cases}
\frac{d-2+\sqrt{d^2+12}}{2}& \mbox{if $d$ is even}\\
\frac{d-2+\sqrt{d^2+8}}{2}& \mbox{if $d$ is old}.
\end{cases}
$$
Then $\kappa(G) \geq 2$.
\end{prop}




\begin{prop}{\rm\cite[Abiad et al.]{abi-bri-mar-O}}
Let $k$ be an integer and $G$ be a $d$-regular graph of order $n$ with $d \geq k \geq 2$. Set $f(d,k) =d+1$ if $k=2$ and $f(d,k) =d+2-k$ otherwise.
If $\lambda_2(G) < d-\frac{(k-1)dn}{2f(d,k)(n-f(d,k))}$, then $\kappa(G) \geq k$.
\end{prop}

\begin{prop}{\rm\cite[Liu]{liu-AMC}}
Let $k$ be an integer and $G$ be a $d$-regular graph of order $n$ with $d \geq k \geq 2$. Set $\beta = \lceil \frac{1}{2}(d+1+\sqrt{(d+1)^2}-2(k-1)d)\rceil$ and
$$\varphi(d,k) =
\begin{cases}
(d+1)(n-d -1)& \mbox{if $k=2$}\\
(d-k + 2)(n-d+k-2)& \mbox{if $k \geq 3$ and $d \leq 2k-4$}\\
\beta(n-\beta)& \mbox{if $k \geq 3$ and $d > 2k-4$}.
\end{cases}
$$
If $\lambda_2(G) < d- \frac{(d-1)nd}{2\varphi(d,k)}$, then $\kappa(G) \geq k$.
\end{prop}

Hong et al. \cite{hong-xia-lai} found their bound for $\lambda_2(G)$ improving the previous two ones.

\begin{prop}{\rm\cite[Hong et al.]{hong-xia-lai}}
Let $k$ be an integer and $G$ be a $d$-regular graph of order $n$ with $d \geq k\geq 2$. If $\lambda_2(G) < d - \frac{(k-1)nd}{(n-k+1)(k-1)+4(d-k+2)(n-d-1)}$, then $\kappa(G) \geq k$.
\end{prop}

As we've already seen, the above bounds for $\lambda_2(G)$ are not sharp. Then, the aim of the present paper is to find  the best upper bounds on the second largest eigenvalues of regular graphs guaranteeing a desired vertex-connectivity. In other words, we investigate the challenging question asked by Mohar and O (private communication) and  alluded to briefly by Cioab\u{a} \cite{cio-LAA,cio-CMJ} and raised formally by Abiad et al. \cite{abi-bri-mar-O}.

\begin{prob}
For a $d$-regular simple graph or multigraph $G$ and for $1 \leq t \leq d-1$, what is the best upper bound for $\lambda_2(G)$ which guarantees that $\kappa(G) \geq  t + 1$ or that $\kappa'(G) \geq  t+1$?
\end{prob}

Remark, for the edge-connectivity $\kappa'(G) $, that Cioab\u{a} \cite{cio-LAA} proved the cases $t = 1,2$ in 2010 and O et al. \cite{O-arxiv}  settled down for all $t \geq 3$ in 2018. However, it seems unlikely to solve this problem for vertex-connectivity once and for all.  As a starting point, we capture this problem in the case $t =1$, and characterize all families of extremal graphs. 

To describe our results, we introduce some notations and terminology. As usual, let  $C_n$, $K_n$ and $M_n$ be  the {\it cycle}, the {\it complete graph} and the {\it perfect matching} with order $n$.  The {\it sequential join} $G_1 \vee \cdots \vee G_k $ of graphs $G_1,\ldots,G_k$ is the graph formed by taking one copy of each graph and adding additional edges from each vertex of $G_i$ to all vertices of $G_{i+1}$, for $1 \leq i \leq k-1$.

For $S, T \subseteq V(G)$, let $E(S, T) = \{(u, v)|u \in S, v \in T, (u, v) \in E(G)\}$ be the set of edges from $S$ to $T$. The graph $G-S$ is derived from $G$ by deleting the vertices of $S$ and the edges incident with the vertices in $S$. If $S=\left\{v\right\}$, we denote by $G-v$ for short. For two integers $d\geq3$ and $1\leq c\leq d-1$, we define the following set $\mathscr{G}_{d,c}$ and graph $G_{d,c}$ relating to our main results.

\begin{itemize}
\item[$\bullet$]
$\mathscr{G}_{d,c}$ is the set of connected $d$-regular graphs $G$ with a cut vertex (say $u$) such that $G-u$ has a component $G_1$ with $|E(u,V(G_1))|=c$.
\item[$\bullet$]
Let $d\geq3$ be odd. Define
$$ G_{d,c}=
\begin{cases}
K_2 \vee \overline{M_{d-1}} \vee K_1 \vee K_1 \vee \overline{M_{d-1}} \vee K_2 & \mbox{if $c=1$ or $c=d-1$};\\
\overline{M_{d+2-c}} \vee \overline{\mathcal{C}_{c}} \vee K_1 \vee \overline{M_{d-c}} \vee K_{c+1}& \mbox{if $c \in (2, d-2]$ is odd};\\
K_{d+1-c} \vee \overline{M_{c}} \vee K_1 \vee \overline{\mathcal{C}_{d-c}} \vee \overline{M_{c+2}}& \mbox{if $c \in [2, d-2)$ is even},
\end{cases}
$$
where $\mathcal{C}_c = C_{c_1} \cup \cdots \cup C_{c_s}$ is the union of disjoint cycles $C_{c_i}$ and $\sum_{i=1}^sc_i = c$.
\item[$\bullet$]
Let $d\geq4$ be even. Clearly, $c \in [2,d-2]$ is even, and define $$G_{d,c}=K_{d+1-c} \vee \overline{M_{c}} \vee K_1 \vee \overline{M_{d-c}} \vee K_{c+1}.$$
\end{itemize}
Remark that $G_{d,1}=G_{d,d-1}$ and $G_{d,c} \in \mathscr{G}_{d,c}$.

We are now in the stage to present the main result of this paper.

\begin{thm}\label{main2}
Let $d\geq3$ and $G$ be a connected $d$-regular graph.
\begin{itemize}
\item[{\rm (i)}]
Let $d=3$. If $\lambda_{2}(G)\leq\lambda_{2}(G_{d,1})$ and $G \neq G_{d,1}$, then $\kappa(G) \geq 2$.
\item[{\rm (ii)}]
Let $d \geq 4$ be even. If $\lambda_{2}(G)\leq\lambda_{2}(G_{d,2\lfloor\frac{d}{4}\rfloor})$ and $$G \neq G_{d,2\lfloor\frac{d}{4}\rfloor} = K_{2\lfloor\frac{d+2}{4}\rfloor+1} \vee \overline{M_{2\lfloor\frac{d}{4}\rfloor}} \vee K_1 \vee M_{2\lfloor\frac{d+2}{4}\rfloor} \vee K_{2\lfloor\frac{d}{4}\rfloor+1},$$ then $\kappa(G) \geq 2$.
\item[{\rm (iii)}]
Let  $d \geq 5$ be odd and $\mathcal{C}_c$ be defined as above.
\begin{itemize}
\item[{\rm (a)}]
For odd $\frac{d-1}{2}$, if $\lambda_{2}(G)\leq\lambda_{2}(G_{d,\frac{d-1}{2}})$ and $$G \neq G_{d,\frac{d-1}{2}}=\overline{M_{\frac{d+5}{2}}} \vee \overline{\mathcal{C}_{\frac{d-1}{2}}} \vee K_1 \vee \overline{M_{\frac{d+1}{2}}} \vee K_{\frac{d+1}{2}},$$ then $\kappa(G) \geq 2$.
\item[{\rm (b)}]
For even $\frac{d-1}{2}$, if $\lambda_{2}(G)\leq\lambda_{2}(G_{d,\frac{d-1}{2}})$ and $$G \neq G_{d,\frac{d-1}{2}}= K_{\frac{d+3}{2}} \vee \overline{M_{\frac{d-1}{2}}} \vee K_1 \vee \overline{\mathcal{C}_{\frac{d+1}{2}}} \vee \overline{M_{\frac{d+3}{2}}},$$ then $\kappa(G) \geq 2$.
\end{itemize}
\end{itemize}
\end{thm}

\section{Preparation}

Let $M_n(F)$ be the set of all $n$ by $n$ matrices over a field $F$. A matrix $M = [m_{ij}] \in M_n(F)$ is {\it tridiagonal} if $m_{ij} = 0$, whenever $|i - j| > 1$.

\begin{lem}{\rm \cite{bro-coh-neu,horn-joh}}\label{tridiagonal matrix}
Let $M$ be a non-negative tridiagonal matrix as follows:
\begin{center}$A=$
$\left(\begin{array}{ccccc}
 a_{0}&b_{0}&&&0\\
 c_{1}&a_{1}&b_{1}&&\\
 &\ddots&\ddots&\ddots&\\
 & &\ddots&\ddots&b_{n-1}\\
 0 & &&c_{n}&a_{n}
  \end{array}\right)$,
   \end{center}
Assume each row sum of $M$ equals $d$. If $M$ has eigenvalues $\lambda_{1}=d,\lambda_{2},...,\lambda_{n+1}$ indexed in non-increasing order, then the $n\times n$ matrix
\begin{center}$\widetilde{A}=$
$\left(\begin{array}{ccccc}
 d-b_{0}-c_{1}&b_{1}&&&0\\
 c_{1}&d-b_{1}-c_{2}&b_{2}&&\\
 &c_{2}&\ddots&\ddots&\\
 & &\ddots&\ddots&b_{n-1}\\
 0 & &&c_{n-1}&d-b_{n-1}-c_{n}
  \end{array}\right)$
   \end{center}
has eigenvalues $\lambda_{2},\lambda_{3},...,\lambda_{n+1}$.
\end{lem}

For the adjacency matrix $A(G)$, an {\it equitable partition} of a graph $G$ is a partition of the vertex set $V(G)$ into parts $V_i$ such that each vertex in $V_i$ has the same number $b_{i,j}$ of neighbors in part $V_j$ for any $j$. Then the matrix $B = (b_{i,j})$ is called the {\it quotient matrix} of $G$ w.r.t. the given partition.

\begin{lem}{\rm \cite{bro-hae}}\label{interlacing}
Let $B$ be the quotient matrix of a graph $G$ w.r.t. the partition $\{V_1,\ldots,V_m\}$. Then the eigenvalues of $B$ interlace the eigenvalues of $G$. Moreover, if this partition is equitable, then each eigenvalue (with multiplicity) of $B$ is an eigenvalue of $G$.
\end{lem}

Look back to the graph $G_{d,c} \in \mathscr{G}_{d,c}$. Obviously, $G_{d,c}$ contains $2d+4$ vertices for odd $d\geq3$ and contains $2d+3$ vertices for even $d\geq4$. Moreover, $G_{d,c}=G_{d,d-c}$. Note that $G_{d,1}=G_{d,d-1}$ is the graph $X_d$ defined in \cite[(4)]{cio-LAA} and used in \cite[Theorem 1.4]{cio-LAA}, the proof of which implies the following conclusion.

\begin{lem}{\rm \cite{cio-LAA}}\label{cut edge}
Let $d\geq3$ be an odd integer. Then $\lambda_{2}(G_{d,1})(=\lambda_{2}(G_{d,d-1}))$ is the largest root of $f_0(x) = 0$ with $f_0(x) = x^{3}-(d-3)x^{2}-(3d-2)x-2$. Moreover, if $G$ is a $d$-regular graph with cut edges, then $\lambda_{2}(G)\geq\lambda_{2}(G_{d,1})$ with equality if and only if $G=G_{d,1}$.
\end{lem}

\begin{lem}\label{G(ka)}
Let $G_{d,c}$ be the graph defined as above.
\begin{itemize}
\item[{\rm (i)}]
Let $d\geq4$ and $2\leq c\leq d-2$ be even integers. Then $\lambda_{2}(G_{d,c})$ is the largest root of equation $f_1(x) = 0$, where $f_1(x) = x^{4}-(d-4)x^{3}-(4d-4)x^{2}+(2cd-2c^{2}-4d)x+3c(d-c)$.
\item[{\rm (ii)}]
Let $d\geq5$ be an odd integer and $2\leq c\leq d-2$. Then $\lambda_{2}(G_{d,c})$ is the largest root of equation $f_2(x) = 0$, where $f_2(x) =x^{4}-(d-5)x^{3}-(5d-6)x^{2}+(2cd-2c^{2}-6d)x+4c(d-c)$.
\end{itemize}
\end{lem}

\begin{proof}
For (i),  $G_{d,c}=K_{d+1-c} \vee \overline{M_c} \vee K_1 \vee \overline{M_{d-c}} \vee K_{c+1}$ has an equitable partition $V(G_{d,c}) = V(K_{d+1-c}) \cup V(\overline{M_c}) \cup V(K_1) \cup V(\overline{M_{d-c}}) \cup V(K_{c+1})$. Then the corresponding quotient matrix is
$$B_{1}=
\left(\begin{array}{ccccc}
d-c&c&0&0&0\\
d+1-c&c-2&1&0&0\\
0&c&0&d-c&0\\
0&0&1&d-c-2&c+1\\
0&0&0&d-c&c
\end{array}\right),
$$
which, along with  Lemma \ref{interlacing}, shows that the eigenvalues of $B_1$ are  the eigenvalues of $G_{d,c}$.
By Lemma \ref{tridiagonal matrix} we have $\lambda_{2}(B_1)=\lambda_1(\widetilde{B_1})$, where
$$\widetilde{B_1}=
\left(\begin{array}{cccc}
-1&1&0&0\\
d+1-c&d-c-1&d-c&0\\
0&c&c-1&c+1\\
0&0&1&-1
\end{array}\right).
$$
Thus, a direct calculation shows that $\lambda_{1}(\widetilde{B_{1}})$ is the largest root of equation  $f_1(x) = 0$.
Note $\lambda_{1}(\widetilde{B_{1}})>d-\frac{1}{2}.$ Let $W$ be the characteristic space of the five parts (of the equitable partition) of $G_{d,c}$. Clearly, the dimension of $W$  is five. Note that each eigenvector ${\bf x}$ associated to an eigenvalue of $B_{1}$ can be extended to be an eigenvector ${\bf x'}$ of $G_{d,c}$, and that the components of ${\bf x'}$
corresponding to each part of the equitable partition are the same. Hence, the five independent eigenvectors extended by those of $B_{1}$ span $W$.
Except eigenvalues of $B_{1}$, remark again that the eigenvectors of other eigenvalues of $G_{d,c}$  are orthogonal to those in $W$.
So, other eigenvalues of $G_{d,c}$ are also the eigenvalues of graph $K_{d+1-c}\cup\overline{M_c}\cup K_{1}\cup\overline{M_{d-c}}\bigcup K_{c+1}$, whose maximum degree is not larger than $d-1$. Hence, other eigenvalues of $G_{d,c}$ except the ones of $B_{1}$ are less than $d-1$. Thereby, $\lambda_{2}(G_{d,c})=\lambda_{2}(B_{1})=\lambda_{1}(\widetilde{B_{1}})$, which is the largest root of equation $f_1(x)=0$.\smallskip

For (ii), consider firstly $ c \in [2, d-2]$ to be an odd integer. Clearly, $G_{d,a}=\overline{M_{d+2-c}} \vee \overline{\mathcal{C}_{c}} \vee K_1 \vee \overline{M_{d-c}} \vee K_{c+1}$ has an equitable partition $V(\overline{M_{d+2-c}}) \cup V(\overline{\mathcal{C}_{c}}) \vee V(K_1) \vee V(\overline{M_{d-c}}) \vee V(K_{c+1})$. Then, the related quotient matrix  is
$$
B_{2}=\left(\begin{array}{ccccc}
d-c&c&0&0&0\\
d+2-c&c-3&1&0&0\\
0&c&0&d-c&0\\
0&0&1&d-c-2&c+1\\
0&0&0&d-c&c
\end{array}\right).
$$
From Lemma \ref{interlacing} the eigenvalues of $B_{2}$ are also the eigenvalues of $G_{d,c}$. By Lemma \ref{tridiagonal matrix} again we get $\lambda_{2}(B_2)=\lambda_{1}(\widetilde{B_2})$, where
\begin{equation}\label{B2}
\widetilde{B_2}=
\left(\begin{array}{cccc}
-2&1&0&0\\
d+2-c&d-c-1&d-c&0\\
0&c&c-1&c+1\\
0&0&1&-1
\end{array}\right).
\end{equation}
By a routine computing, $\lambda_{1}(\widetilde{B_{2}})$ is the largest root of equation $f_2(x) =0$. Note that $\lambda_{1}(\widetilde{B_{2}})>d-\frac{1}{2}$. Similarly to (i), we can verify that $\lambda_{2}(G_{d,c})=\lambda_{2}(B_{2})=\lambda_{1}(\widetilde{B_{2}})$ is the largest root of equation $f_2(x)=0$. If $2\leq c \leq d-2$ is even, then $2\leq d-c\leq d-2$ is odd. Due to $G_{d,c}=G_{d,d-c}$, applying the above discussion to $G_{d,d-c}$ we obtain the conclusion.
\end{proof}

\begin{lem}\label{compare}
For $d\geq3$ and $1\leq c\leq k-1$, let $G_{k,a}$ be the graph defined above.
\begin{itemize}
\item[{\rm (i)}]
If $d\geq4$ is even, then $\lambda_{2}(G_{d,2})>\lambda_{2}(G_{d,4})>\cdots>\lambda_{2}(G_{d,2\lfloor\frac{d}{4}\rfloor})$.
\item[{\rm (ii)}]
If $d\geq3$ is odd, then $\lambda_{2}(G_{d,1})>\lambda_{2}(G_{d,2})>\cdots>\lambda_{2}(G_{d,\frac{d-1}{2}})$.
\end{itemize}
\end{lem}

\begin{proof}
Let $d\geq4$ be even and $2\leq c\leq2\lfloor\frac{d}{4}\rfloor$. By Lemma \ref{G(ka)}(i), $\lambda_{2}(G_{k,a})$ is the largest root of equation $f_1(x)=0$. Since
\begin{equation*}
\begin{split}
f_1(x)&=x^{4}-(d-4)x^{3}-(4d-4)x^{2}+(2cd-2d^{2}-4k)x+3c(d-c)\\
&=x^{4}-(d-4)x^{3}-(4d-4)x^{2}-4dx+c(d-c)(2x+3),
\end{split}
\end{equation*}
then for $x>0$ it is easy to see that $c(d-c)(2x+3)$ is strictly increasing w.r.t. $c$. Thereby, the largest root of equation $f_1(x)=0$ is strictly decreasing w.r.t. $c$, which implies $\lambda_{2}(G_{d,2})>\lambda_{2}(G_{d,4})>\cdots>\lambda_{2}(G_{d,2\lfloor\frac{d}{4}\rfloor})$.

We next show (ii). For $d=3$, the result is trivial. Let $d\geq5$ be odd and $2\leq c \leq\frac{k-1}{2}$. By Lemma \ref{G(ka)}(ii), $\lambda_{2}(G_{d,c})$ is the largest root of equation $f_2(x)=0$. Since
\begin{equation*}
\begin{split}
f_2(x)&=x^4-(d-5)x^3-(5d-6)x^2+(2cd-2c^2-6d)x+4c(d-c)\\
      &=x^4-(d-5)x^3-(5d-6)x^2-6dx+2c(d-c)(x+2),
\end{split}
\end{equation*}
then for $x>0$ we get that $2c(d-c)(x+2)$ is strictly increasing w.r.t. $c$. Therefore, the largest root of equation $f_2(x)=0$ is  strictly decreasing w.r.t. $c$, and hence $\lambda_2(G_{d,2})>\lambda_2(G_{d,3})>\cdots>\lambda_2(G_{d,\frac{d-1}{2}})$.

For $c=1$, from Lemma \ref{cut edge} it follows that $\lambda_{2}(G_{d,1})$ is the largest root of $f_0(x)=0$.  By Lemma \ref{G(ka)}(ii) again, $\lambda_{2}(G_{d,2})$ is the largest root of equation
\begin{equation*}
\begin{split}
f_2(x)|_{c=2}&=x^4-(d-5)x^3-(5d-6)x^2-(2d+8)x+8d-16\\
      &=(x+2)(x^{3}-(d-3)x^{2}-(3d-2)x-2)+2x(2d-5-x)+8d-12\\
      &=0.
\end{split}
\end{equation*}
For $d\geq5$ and  $\lambda_{2}(G_{k,1}) \leq x < d$, we get $f_0(x) \geq0$, and then $x^4-(d-5)x^3-(5d-6)x^2-(2d+8)x+8d-16>0$. Hence, $\lambda_{2}(G_{d,1})>\lambda_{2}(G_{d,2})$, and so $\lambda_{2}(G_{k,1})>\lambda_{2}(G_{k,2})>\cdots>\lambda_{2}(G_{k,\frac{k-1}{2}})$.

This completes the proof.
\end{proof}

\section{Proof of Theorem \ref{main2}}

Recall, for two integers $d\geq3$ and $1\leq c \leq d-1$, that the family $\mathscr{G}_{d,c}$ is defined in Section 1.

\begin{lem}\label{lem-1}
Let $G \in \mathscr{G}_{d,c}$.  If $d\geq3$ is odd  and $c=1$ or $c=d-1$, then $\lambda_{2}(G)\geq\lambda_{2}(G_{d,c})$ with equality holds if and only if $G=G_{d,c}$.
\end{lem}

\begin{proof}
In this case, $G$ has cut edges. Then the result follows from Lemma \ref{cut edge}.
\end{proof}

\begin{lem}\label{lem-2}
Let $G \in \mathscr{G}_{d,c}$.  If $d\geq4$ and $c \in [2, d-2]$ are even , then $\lambda_{2}(G)\geq\lambda_{2}(G_{d,c})$, where he equality holds if and only if $G=G_{d,c}$.
\end{lem}

\begin{proof}
Since $G \in \mathscr{G}_{d,c}$, then there exists a cut vertex $u$ such that $|E(u,V(G_1))|=c$, where $G_1$ is a component of $G-u$. Let $G_2$ be the union of
other components of $G-u$. Then $|E(u,V(G_2))|=d-c$. If $|V(G_1)|\leq d$, then $c=|E(u,V(G_1))|\geq |V(G_1)|(d+1-|V(G_1)|)\geq d$, a contradiction. Hence,
$|V(G_1)|\geq d+1$, and so is $|V(G_2)|$ similarly. Let $V_1$ and $V_2$ be the sets of the neighbors of $u$ in $G_1$ and in $G_2$ respectively. Then $|V_1|=c$
and $|V_2|=d-c$. Set $G_1^{'}=G_1-V_1$ and $G_2^{'}=G_2-V_2$. Thus, $|V(G_1^{'})|= p \geq d+1-c$ and $|V(G_2^{'})|=q \geq d+1-(d-c)=c+1$. After putting
$|E(V(G_1^{'}),V_1)|=r$ and $|E(V(G_2^{'}),V_2)|=t$, we obtain a partition of $G$ with $V(G) = V(G_1^{'}) \cup V_1 \cup \{u\} \cup V_2 \cup V(G_2^{'})$ whose quotient matrix is
\begin{equation}\label{B3}
B_{3}=\left(\begin{array}{ccccc}
d-\frac{r}{p}&\frac{r}{p}&0&0&0\\
\frac{r}{c}&d-1-\frac{r}{c}&1&0&0\\
0&c&0&d-c&0\\
0&0&1&d-1-\frac{t}{d-c}&\frac{t}{d-c}\\
0&0&0&\frac{t}{q}&d-\frac{t}{q}
\end{array}\right).
\end{equation}
Using Lemmas \ref{tridiagonal matrix} and \ref{interlacing} we get $\lambda_{2}(G)\geq \lambda_{2}(B_{3})=\lambda_{1}(\widetilde{B_{3}})$, where $\widetilde{B_{3}}$ is
\begin{equation}\label{B33}
\left(\begin{array}{cccc}
d-\frac{r}{p}-\frac{r}{c}&1&0&0\\
\frac{r}{c}&d-c-1&d-c&0\\
0&c&c-1&\frac{t}{d-c}\\
0&0&1&d-\frac{t}{d-c}-\frac{t}{q}
\end{array}\right).
\end{equation}

\noindent{\bf Fact 1.}  The largest root $\lambda_{1}(\widetilde{B_{3}})$ of matrix $\widetilde{B_{3}}$ is strictly decreasing w.r.t. $r$ and $t$.\smallskip

\noindent{\it Proof of Fact 1.} We only prove the conclusion for $r$, as it is similar for $t$. By a computing, the characteristic polynomial of $\widetilde{B_{3}}$ is equal to $h(x)=|x I-\widetilde{B_{3}}|=(x-d)h_1(x)+r h_2(x)$, where $h_2(x)=(\frac{1}{p}+\frac{1}{c})h_1(x)+\frac{t}{c(d-c)}-\frac{1}{c}(x+c(d-1-c)+1)(x-d+\frac{t}{d-c}+\frac{t}{q})$, and $h_1(x)$ is the next characteristic polynomial of a principle sub-matrix of $\widetilde{B_{3}}$
$$\left(\begin{array}{ccc}
d-c-1&d-c&0\\
c&c-1&\frac{t}{d-c}\\
0&1&d-\frac{t}{d-c}-\frac{t}{q}
\end{array}\right).
$$
Clearly, for $x \geq \lambda_1(\widetilde{B_{3}})$ we get $h_1(x)>0$ and  $h(x)\geq0$. Thus, for $\lambda_{1}(\widetilde{B_{3}}) \leq x <d$ we arrive at  $h_2(x)>0$ and hence $(x-d)h_1(x)+r_1h_2(x) > h(x) \geq 0$ for any $r_1>r$. Thereby, the largest root of equation $(x-d)h_1(x)+r_1h_2(x)=0$ is less than $\lambda_1(\widetilde{B_{3}})$. So, the largest root of $\widetilde{B_{3}}$ is strictly decreasing w.r.t. $r$. \hfill{$\Box$}

Employing Fact $1$, we can make $r$ and $t$ as large as possible. By the well-know Perron-Frobenius Theorem, the largest eigenvalue of an irreducible matrix will strictly decrease if its positive elements decrease. Then for given $r$ and $t$ we set $p$ and $q$ as small as possible in $\widetilde{B_{3}}$. Noting $r\leq c(d-1)$ and $t\leq(d-c)(d-1)$, we suppose $d+1-c \leq p\leq d-1$, $c+1\leq q\leq d-1$, $r=cp$ and $t=(d-c)q$ (For example, if $p\geq d$, we can firstly make $r$ as large as possible, i.e., $r=c(d-1)$. And then we set $p$ as small as possible, i.e., $p= d-1$). Thereby, $G[V_1],G[V(G_1^{'})],G[V_2]$ and $G[V(G_2^{'})]$ are four regular graphs of degrees $d-1-p,d-c,d-1-q$ and $c$, respectively (note that for any two positive integers $d<t$,  then there exists a $d$-regular graph on $t$ vertices if $d$ is even). Apparently, $\widetilde{B_{3}}$ is equal to the following matrix $D$
\begin{equation}\label{D}
D=\left(\begin{array}{cccc}
d-c-p&1&0&0\\
p&d-c-1&d-c&0\\
0&c&c-1&q\\
0&0&1&c-q
\end{array}\right).
\end{equation}

\noindent{\bf \textbf{Fact} 2}.  The largest root $\lambda_{1}(D)$ of matrix $D$ is a strictly increasing function w.r.t. $p$ and $q$.\smallskip

\noindent{\it Proof of Fact 2.} We only prove the conclusion for $p$, as it is similar for $q$. A straightforward calculation yields the characteristic polynomial of $D$ is equal to $g(x)=|x I-D|=(x-d+c)g_1(x)+pg_2(x)$, where $g_2(x)=g_1(x)-(x-c+1)(x-c+q)$, and $g_1(x)$ is the characteristic polynomial of next principle sub-matrix of $D$,
$$
\left(\begin{array}{ccc}
d-c-1&d-c&0\\
c&c-1&q\\
0&1&c-q
\end{array}\right).
$$
For $x \geq \lambda_{1}(D)$, clearly we get $g_1(x)>0$.   For any $d+1-c\leq p^{'}<p$, we now prove $(x-d+c)g_1(x)+p^{'} g_2(x)>0$ whenever $x \geq \lambda_{1}(D)$. Evidently, for any $ \lambda_{1}(D) \leq x < d$, if $g_2(x)\geq0$, then $(x-d+c)g_1(x)+p^{'}g_2(x)>0$. If $g_2(x)<0$ for some $ \lambda_{1}(D)\leq x<d$, then $(x-d+c)g_1(x)+p^{'}g_2(x)>(x-d+c)g_1(x)+pg_2(x)\geq0$. Hence, $(x-d+c)g_{1}(x)+p^{'}g_2(x)>0$ for $x \geq \lambda_{1}(D)$, which implies that the largest root of equation $(x-d+c)g_1(x)+p^{'}g_2(x)=0$ is less than $\lambda_{1}(D)$. So, the largest root $\lambda_{1}(D)$ of matrix $D$ is strictly increasing w.r.t. $p$.\hfill$\Box$

In view of Fact $2$, we can set $p=d+1-c$ and $q=c+1$ which together with $r=cp$ and $t=(d-c)q$ leads to $\widetilde{B_{3}}=\widetilde{B_{1}}$ (defined in Lemma \ref{G(ka)}). Consequently, $\lambda_{2}(G) \geq \lambda_{2}(G_{d,c})$ with equality holds if and only if $G=G_{d,c}$.
\end{proof}

For the completeness of next lemma,  we give an imitate proof with the similar as Lemma \ref{lem-2}.

\begin{lem}\label{lem-3}
Let $G \in \mathscr{G}_{d,c}$ and $d\geq5$ be an odd integer.
\begin{itemize}
\item[{\rm (i)}]
If $c \in [2,d-2]$ is odd, then $\lambda_{2}(G)\geq\lambda_{2}(G_{d,c})$ with equality if and only if $$G=\overline{M_{d+2-c}} \vee \overline{\mathcal{C}_c} \vee K_1 \vee \overline{M_{d-c}} \vee K_{c+1},$$ where $\mathcal{C}_c$ is the union of disjoint cycles on $c$ vertices.
\item[{\rm (ii)}]
If $c \in [2, d-2]$ is even, then $\lambda_{2}(G)\geq\lambda_{2}(G_{d,c})$ with equality  if and only if $$G=K_{d+1-c} \vee \overline{M_{c}} \vee K_1 \vee \overline{\mathcal{C}_{d-c}} \vee \overline{M_{c+2}},$$ where $\mathcal{C}_{d-c}$ is the union of disjoint cycles on $d-c$ vertices.
\end{itemize}
\end{lem}

\begin{proof}
We only need to show (i). Otherwise, if $c \in [2, d-2]$ is even, then $d-c$ is odd. Due to $G_{d,c}=G_{d,d-c}$, then the proof is similar to the case in which $c$ is odd.

Since $G \in \mathscr{G}_{d,c}$, then there exists a cut vertex $u$ such that $|E(u,V(G_1))|=c$, where $G_1$ is a component of $G-u$. Let $G_2$ be the union of
other components of $G-u$. Then $|E(u,V(G_2))|=d-c$. Similarly to Lemma \ref{lem-2}, we have $|V(G_1)|\geq d+1$ and $|V(G_2)|\geq d+1$. Since $d|V(G_1)|=2|E(G_1)|+c$ and $d|V(G_2)|=2|E(G_2)|+d-c$, then $|V(G_1)|$ is odd, and thus $|V(G_1)|\geq d+2$ and $|V(G_2)|$ is  even.  Let $V_1$ and $V_2$ be the sets of the neighbors of $u$ in $G_1$ and in $G_2$ respectively. Then $|V_1|=c$
and $|V_2|=d-c$. Set $G_1^{'}=G_1-V_1$ and $G_2^{'}=G_2-V_2$. Thus, $|V(G_1^{'})|= p \geq d+2-c$ and $|V(G_2^{'})|=q \geq d+1-(d-c)=c+1$. Set $|E(V(G_1^{'}),V_1)|=r$ and $|E(V(G_2^{'}),V_2)|=t$. We obtain a partition of $G$ with $V(G) = V(G_1^{'}) \cup V_1 \cup \{u\} \cup V_2 \cup V(G_2^{'})$ whose quotient matrix is just $B_3$ in \eqref{B3}. By Lemmas \ref{tridiagonal matrix} and \ref{interlacing}, we have $\lambda_{2}(G)\geq\lambda_{2}(B_{3})=\lambda_{1}(\widetilde{B_{3}})$, where $\widetilde{B_{3}}$ is defined in \eqref{B33}. As shown in Lemma \ref{lem-2}, by Fact 1 we get that the largest root $\lambda_{1}(\widetilde{B_{3}})$ of matrix $\widetilde{B_{3}}$ is strictly decreasing w.r.t. $r$ and $t$.

By the well-known Perron-Frobenius Theorem, the largest eigenvalue of an irreducible matrix will strictly decrease if its positive elements decrease. Then we can make $p$ and $q$ as small as possible, and set $r$ and $t$ as large as possible. Noting $r\leq c(d-1)$ and $t\leq(d-c)(d-1)$, we suppose $d+2-c\leq p\leq d-1$, $c+1\leq q\leq d-1$, $r=cp$ and $t=(d-c)q$. Then $G[V(G_1)],G[V(G_2^{'})],G[V_2],G[V(G_2^{'})]$ are four regular graphs of degree $d-1-p,d-c,d-1-q$ and $c$, respectively (note that such graphs exist since $p$ and $q$ are even). Thus, $\widetilde{B_{3}}$ is equal to matrix $D$ defined in \eqref{D}.

From Fact $2$ it follows that the largest root $\lambda_{1}(D)$ is strictly increasing w.r.t $p$ and $q$. Thus, we can set $p=d+2-c$ and $q=c+1$, which implies $\widetilde{B_3}=\widetilde{B_{2}}$ defined in \eqref{B2}. Consequently, $\lambda_{2}(G)\geq\lambda_{2}(G_{d,c})$, and equality holds if and only if $G=\overline{M_{d+2-c}} \vee \overline{\mathcal{C}_c} \vee K_1 \vee \overline{M_{d-c}} \vee K_{c+1}$, where $\mathcal{C}_c$ is the union of disjoint cycles on $c$ vertices.
\end{proof}

\noindent{\bf Proof of Theorem \ref{main2}}.
Clearly, (i) follows from Lemma \ref{cut edge}. For (ii) and (ii), assume that $G$ has a cut vertex, say $u$. Then, there exists some component $G_1$ of $G-u$ such that $|E(u,V(G_1))| = c \leq \frac{d-1}{2}$ if $d$ is odd, or such that $c \leq 2 \lfloor\frac{d}{4}\rfloor$ if $d$ is even. From Lemmas \ref{lem-1}--\ref{lem-3} and  \ref{compare} we get  $\lambda_{2}(G)\geq\lambda_{2}(G_{d,\frac{d-1}{2}})$ with the equality iff $G = G_{d,\frac{d-1}{2}}$ when $d$ is odd, or  $\lambda_{2}(G)\geq\lambda_{2}(G_{d,2\lfloor\frac{d}{4}\rfloor})$ with the equality iff $G =G_{d,2\lfloor\frac{d}{4}\rfloor}$ when $d$ is even, a contradiction.

Consequently, $\kappa(G) \geq 2$. \hfill$\Box$

\section*{Acknowledgments}
The authors are supported for this research by the National Natural Science Foundation of China (No. 11971274).

\end{document}